\documentclass[12pt]{amsart}

\usepackage{ucs}

\usepackage{amssymb}
\usepackage{amsthm}
\usepackage{amsmath}
\usepackage{latexsym}
\usepackage[cp1251]{inputenc}
\usepackage{graphicx}
\usepackage{wrapfig}
\usepackage{caption}
\usepackage{subcaption}
\usepackage{indentfirst}
\usepackage[left=2.6cm,right=2.6cm,top=2.6cm,bottom=2.6cm,bindingoffset=0cm]{geometry}
\usepackage{enumerate}
\usepackage{makecell}
\usepackage{float}
\usepackage{lipsum}
\usepackage{algorithmic}

\usepackage{fullpage}

\usepackage{amsfonts}

\usepackage[ruled]{algorithm2e}
\usepackage{mathtools}
\usepackage{dsfont}
\usepackage{rotating}

\usepackage{floatflt,epsfig}
\usepackage{lineno,hyperref}
\usepackage{longtable}

\DeclareMathOperator{\cay}{Cay}

\makeatletter 
\def\@seccntformat#1{\csname the#1\endcsname. } 
\def\@biblabel#1{#1.}

\makeatother

\title{On Neumaier Cayley graphs}

\author{Rhys J. Evans}

\address{\parbox{\linewidth}{Rhys J. Evans\\ School of Mathematics and Statistics, University of Sydney\\ Camperdown, Sydney, NSW 2006, Australia}}

\email{rhysjohn.evans@sydney.edu.au}

\author{Sergey Goryainov}

\address{\parbox{\linewidth}{Sergey Goryainov\\ School of Mathematical Sciences, Hebei International Joint Research Center for Mathematics and Interdisciplinary Science, Hebei Key Laboratory of Computational Mathematics and Applications, Hebei Workstation for Foreign Academicians, Hebei Normal University, Shijiazhuang 050024, P.R. China}}

\email{sergey.goryainov3@gmail.com}

\author{Grigory Ryabov}
\address{\parbox{\linewidth}{Grigory Ryabov\\ School of Mathematical Sciences, Hebei Key Laboratory of Computational Mathematics and Applications, Hebei Normal University, Shijiazhuang 050024, P. R. China}}
\address{\parbox{\linewidth}{Sobolev Insitute of Mathematics, Novosibirsk, Russia}}
\email{gric2ryabov@gmail.com}

\author{Da Zhao}
\address{\parbox{\linewidth}{Da Zhao\\ School of Mathematics, East China University of Science and Technology, 130 Meilong Road, Shanghai 200237, China.}}
\email{zhaoda@ecust.edu.cn}

\thanks{}

\date{}

\newtheorem{prop}{Proposition}[section]

\newtheorem{lemm}[prop]{Lemma}
\newtheorem{theo}[prop]{Theorem}
\newtheorem*{ques}{Question}
\newtheorem{corl}[prop]{Corollary}

\theoremstyle{definition}

\newtheorem*{rem}{Remark}

\def\tm#1{\item[{\rm (#1)}]}

\begin{document}
\SetKwComment{Comment}{/* }{ */}

\vspace{\baselineskip}
\vspace{\baselineskip}

\vspace{\baselineskip}

\vspace{\baselineskip}

\begin{abstract}
In the present paper, we study Neumaier Cayley graphs. First, we give a criterion for a Cayley graph to be a Neumaier graph with a spread given by the cosets of a subgroup. Further, we construct a new infinite family of Neumaier Cayley graphs of unbounded nexus. Finally, we provide an algorithm for enumerating Neumaier Cayley graphs and computational results obtained by this algorithm. 
\\
\\
\textbf{Keywords}: Neumaier graphs, Cayley graphs, difference sets.
\\
\\
\textbf{MSC}: 05B10, 05C25, 05E30. 
\end{abstract}

\maketitle

\section{Introduction}

A \emph{regular clique} in a finite regular graph is a clique such that every vertex that does not belong to the clique is adjacent to the same positive number of vertices in the clique. An \emph{edge-regular graph} is a regular graph such that every pair of adjacent vertices has a fixed number of common neighbours. A \emph{Neumaier graph} is a non-complete edge-regular graph containing a regular clique and a \emph{strictly Neumaier graph} is a non-strongly regular Neumaier graph. In \cite{N81}, Neumaier posed the problem of whether there exists a non-complete, edge-regular, non-strongly regular graph containing a regular clique.  

A $k$-regular Neumaier graph $\Gamma$ on $v$ vertices is said to have \emph{parameters} $(v,k,\lambda,m,s)$ if each edge of $\Gamma$ lies in exactly $\lambda$ triangles and $\Gamma$ has an $m$-regular clique of size $s$ (it can be shown that in this case each regular clique in $\Gamma$ is an $m$-regular clique of size $s$). The parameter $m$ is called the \emph{nexus} of $\Gamma$. Since the first construction of strictly Neumaier graphs, there has been an increased interest in the combinatorial and spectral properties of Neumaier graphs (e.g., \cite{ADDK21,ACDKZ23,GT25}).

There are many known families of strictly Neumaier graphs with 1-regular cliques (see \cite{ACDKZ23,GK18,GK19}). Each graph in these families can be seen an instance of a general construction of strictly Neumaier graphs that have a spread of 1-regular cliques \cite{EGKM21}. In contrast, there are only two known infinite families of strictly Neumaier graphs with $m$-regular cliques with $m>1$, each of which consists of a strictly Neumaier graph with a $2^i$-regular clique for every positive integer $i$. Small sporadic examples of strictly Neumaier graphs with $m$-regular cliques have also been computed for $m=1$ \cite{EGKM21,GT25} and $m\leqslant 20$ \cite{GT25}. Therefore, the following question is still unanswered.

\begin{ques} 
For which positive integers $m$ does there exist a strictly Neumaier graph with an $m$-regular clique?
\end{ques}

In this paper, we are interested in Neumaier Cayley graphs and, in particular, in the above question for them. The first result of the paper is an easy criterion for a Cayley graph to be a Neumaier graph with a spread given by the cosets of a subgroup. In the following $H^\#$ denotes the set of nonidentity elements of the group $H$.

\begin{theo}\label{main1}
Let $G$ be a group, $H<G$, and $S\subseteq G$ such that $e\notin S$, $S=S^{-1}$, and $H^\#\subseteq S$. Then $\Gamma=\cay(G,S)$ is a Neumaier graph with a regular clique~$H$ and parameters~$(v,k,\lambda,m,s)$, where $v=|G|$, $k=|S|$, and $s=|H|$, if and only if the following conditions hold:
\begin{itemize}
\tm{1} $|T\cap Hg|=m$ for every $g\in G\setminus H$,

\tm{2} $|Th\cap T|=\lambda-s+2$ for every $h\in H^\#$, 

\tm{3} $|Tg\cap T|=\lambda-2m+2$ for every $g\in T$,
\end{itemize}
where $T=S\setminus H^\#$. 
\end{theo}

It should be noted that the smallest strictly Neumaier graph with parameters $(16,9,4,2,4)$ is isomorphic to a Cayley graph over each of the groups $C_2\times C_8$, $C_2\times D_8$, $D_{16}$, where $C_n$ and $D_n$ are cyclic and dihedral groups of order~$n$, respectively. Several constructions of Neumaier Cayley graphs can be found in~\cite{ACDKZ23,GK18,GT25}. Among them, one can find infinite families of graphs of nexus~1 and several graphs of nexus greater than 1. However, it was not previously known whether there exists an infinite family of strictly Neumaier Cayley graphs of nexus greater than~1. The second result of the paper states the existence of such a family.

\begin{theo}\label{main2}
For every even positive integer $n$, there exists a strictly Neumaier Cayley graph with parameters
$$(2^{2n+2},(2^{n+1}-1)(2^n+1),2(2^n-1)(2^{n-1}+1),2^n,2^{n+1})$$
over any abelian group of order $2^{2n+2}$ having an elementary abelian subgroup of index~$2$.
\end{theo}

As we are able to check, the family of the graphs from Theorem~\ref{main2} is the first known infinite family of strictly Neumaier Cayley graphs of nexus greater than~$1$. The construction of these graphs is based on a usage of the theory of relative difference sets. We also describe a family of Neumaier Cayley graphs of nexus~1 arising from relative difference sets (see Section~\label{s:relative1}).

Finally, the paper contains several computational results. Namely, we provide an algorithm for enumerating Neumaier Cayley graphs with a spread given by the cosets of a subgroup and collect the graphs obtained by a computer calculation using this algorithm.

This paper is organised as follows. In Section~\ref{s:proof1}, we prove Theorem \ref{main1}. A necessary background of relative difference sets is given in Section~\ref{s:rds}. In Section \ref{s:relative}, we provide constructions of Neumaier Cayley graphs based on a usage of relative difference sets. In particular, we prove Theorem~\ref{main2} in this section. In Section \ref{s:enum}, we describe an algorithm for enumerating Neumaier Cayley graphs with a spread given by the cosets of a subgroup and present some computational results. Namely, we list all Neumaier Cayley graphs with a spread given by the cosets of a subgroup and nexus greater than~1 on at most 64 vertices.

\section{Proof of Theorem \ref{main1}}\label{s:proof1}

For completeness, we give a proof of Theorem \ref{main1}. We also give a corollary describing the parameters of a Neumaier graph coming from Theorem \ref{main1} in terms of the corresponding group and connection set.

\begin{proof}[Proof of Theorem~\ref{main1}]
The theorem follows from the fact that $H$ is a clique and each vertex $g\in G$ has neighbourhood $Sg=H^\#g\cup Tg$ in $\Gamma$. 

   For $g\in G\setminus H$ and $h\in H$, $g^{-1}$ and $h$ are adjacent in $\Gamma$ if and only if $hg\in T$. Therefore $g^{-1}$ has $|T\cap Hg|$ neighbours in $H$. It follows that (1) holds if and only if $H$ is $m$-regular. 

    Now assume (1) holds and let $g\in S$. Then $g$ is adjacent to $e$ in $\Gamma$, and shares the neighbours $Sg\cap S$ with $e$. By (1), the shared neighbours are exactly $(Tg\cap T)\cup H^\#$ if $g\in H$, and $(Tg\cap T)\cup (Tg\cap H^\#)\cup (T\cap H^\#g)$ if $g\notin H$. Furthermore, the latter two sets of the second union of sets have size $m-1$ each, and so (2) and (3) hold if and only if $\Gamma$ is edge-regular with parameters $(v,k,\lambda)$ (where we also use the fact that $\Gamma$ is vertex-transitive).
\end{proof}

\begin{corl}
With the notation of Theorem~\ref{main1}, the parameters $v$, $k$, and $\lambda$ can be expressed as follows:
\begin{align*}
    v&=ns\\
    k&=s-1+(n-1)m\\
    \lambda&=s-2+\frac{(n-1)m(m-1)}{(s-1)}
\end{align*}
where $n=|G:H|$.   
\end{corl}

\section{Relative difference sets}\label{s:rds}

In this section, we provide a necessary background of group rings and relative difference sets. We follow~\cite{MR2024} in general. 

Let $G$ be a finite group and $\mathbb{Z}G$ the integer group ring. The identity element of $G$ and the set of all nonidentity elements of $X\subseteq G$ are denoted by $e$ and $X^\#$, respectively. Let $\xi=\sum_{g\in G} a_g g,\eta=\sum_{g\in G} b_g g\in\mathbb{Z}G$. The product of $\xi$ and $\eta$ will be written as $\xi\cdot \eta$. Given $X\subseteq G$, we set $\underline{X}=\sum \limits_{x\in X} {x}\in\mathbb{Z}G$ and $X^{(-1)}=\{x^{-1}:~x\in X\}$. It is easy to check that if $X\subseteq H\leq G$, then 
\begin{equation}\label{trivial}
\underline{H}\cdot \underline{X}=\underline{X}\cdot \underline{H}=|X|\underline{H}.
\end{equation}

Let $N$ be a normal subgroup of $G$. A subset $T$ of $G$ is called a \emph{relative to~$N$ difference set} (\emph{RDS} for short) in $G$ if 
\begin{equation}\label{defrds}
\underline{T}\cdot\underline{T}^{(-1)}=ke+\lambda(\underline{G}-\underline{N}),
\end{equation}
for $k=|T|$ and some positive integer $\lambda$ or, equivalently, if every element in $G\setminus N$ has exactly $\lambda$ representations $t_1t_2^{-1}$  with $t_1,t_2\in T$ and no non-identity element in~$N$ has such a representation. The numbers $(m,n,k,\lambda)$, where $m=|G:N|$ and $n=|N|$, are called the \emph{parameters} of $T$ and the subgroup $N$ is called a \emph{forbidden subgroup}. An RDS $T$ is called \emph{semiregular} (\emph{reversible}, resp.) if $m=k$ ($T=T^{(-1)}$, resp.). If $T$ is semiregular, then $k=m=\lambda n$ and $T$ is a transversal for $N$ in $G$. If $T$ is semiregular and reversible, then $T\cap N=\{e\}$. Reversible semiregular relative difference sets will be called RSRDSs for short further. For a background of RDSs, we refer the reader to the survey~\cite{Pott1996}. 

Given a positive integer $r$, a prime~$p$, and a prime power~$q$, an extraspecial $p$-group of order $p^{2r+1}$ and exponent~$p^2$ and a Heisenberg group of dimension~$(2r+1)$ over a finite field $\mathbb{F}_q$ of order~$q$ are denoted by $E^2_{2r+1}(p)$ and $H_{2r+1}(q)$, respectively.

Let $\mathcal{E}$ be the class consisting of all nontrivial groups of the form 
$$G~\text{and}~G\times C_2^2 \times C_3^{2a} \times C_{p_1}^4 \times \cdots \times C_{p_i}^4,$$
where $G$ is an arbitrary group of the form
$$G=C_4^b \times C_{2^{c_1}}^2 \times \cdots C_{2^{c_j}}^2,$$
$p_1,\ldots,p_i$ are (not necessarily distinct) odd primes, and $a$, $b$, $c_1,\ldots,c_j$ are nonnegative integers. One can check that $\mathcal{E}$ is closed under taking direct products and each group from $\mathcal{E}$ has order of the form $4u^2$ for some positive integer~$u$. By the definition, $\mathcal{E}$ contains a group of order $4u^2$ if and only if $u$ is of the form
$$u=2^r3^sv^2,$$
where $r,s\in\{0,1\}$ and $v$ is an arbitrary nonzero integer.

The lemma below collects families of RSRDSs with $n$ dividing $\lambda$.

\begin{lemm}\label{hadamar}
A group $G$ from the first column of Table~1 has an RSRDS with parameters from the second column of this table. 
\end{lemm} 

\begin{proof}
Families of RSRDSs from Lines~$1$, $2$, $3$ and~$4$ of Table~$1$ can be found in~\cite[Theorem~1.1]{MR2024},~\cite[Theorem~1.2]{MR2024}, \cite[Corollary~4.2]{LM1990}, and~\cite[Theorem~1.4]{CHL2005}, respectively. From~\cite[Theorem~4.2]{JS1997} it follows that each group from $\mathcal{E}$ has a reversible Hadamard difference set. So family of RSRDS from Line~$5$ of Table~$1$ can be constructed by~\cite[Corollary~2.11]{AJP1990}.
\end{proof}

\begin{table}
\centering
{\small
\begin{tabular}{|l|l|l|}
  \hline
  group $G$ & parameters $(m,n,k,\lambda)$ & comment   \\
  \hline
  $H_{2r+1}(q)$ & $(q^{2r},q,q^{2r},q)$ &  $q$ is an odd prime power, $r\geq 1$\\ \hline
  $E^2_{2r+1}(p)$ & $(p^{2r},p,p^{2r},p)$  & $p$ is an odd prime, $r\geq 1$ \\  \hline
  $G_0\times G_0\times C_2^j$ & $(2^{2i},2^j,2^{2i},2^{2i-j})$  &  $G_0$ is abelian, $|G_0|=2^i$, $i\geq j\geq 1$\\  \hline
  
  $C_4^l$ & $(2^{2l},4,2^{2l},2^{2l-2})$  &  $l\geq 3$\\  \hline
  $G_0\times C_2$ & $(4u^2,2,4u^2,2u^2)$  & $G_0\in \mathcal{E}$, $|G_0|=4u^2$ \\  \hline
\end{tabular}
}
\caption{Families of RSRDSs}
\end{table}

We finish this section with the lemma which provides a description of a Cayley graph whose connection set is an RSRDS. It is taken from~\cite[Theorem~1.2(2),~Lemma~3.1]{CL2005}.

\begin{lemm}\label{delta}
Let $G$ be a finite group, $N$ a normal subgroup of $G$, and $T$ an RSRDS in $G$ with forbidden subgroup $N$ and parameters $(m,n,k,\lambda)$. Then the graph $\Delta=\cay(G,T^\#)$ is a distance-regular antipodal graph of diameter~$3$ with intersection array
$$\{m-1,(n-1)m/n,1;1,m/n,m-1\}$$
in which distance between $g_1,g_2\in G$ is equal to~$3$ if and only if $g_2\in Ng_1$. 
\end{lemm}

\section{Infinitely many examples from relative difference sets}\label{s:relative}

In this section, we present two infinite families of Neumaier Cayley graphs constructed from RDSs. The graphs from the first family can also be obtained by application of the construction from~\cite{GK19} to Cayley graphs whose connections sets are RSRDS. The graphs from the second family are exactly the graphs mentioned in Theorem~\ref{main2}. The main tool for the proofs of the results in this section is computations in the group ring. The similar technique was used for studying Deza Cayley graphs in~\cite{BPR2021}.  

\subsection{Construction~$1$}\label{s:relative1}

\begin{lemm}\label{neumaier}
Let $G$ be a finite group, $N$ a normal subgroup of $G$, and $T$ an RSRDS in $G$ with forbidden subgroup $N$ and parameters $(n\lambda,n,n\lambda,\lambda)$. Suppose that $n$ divides $\lambda$. Then for an arbitrary group $U$ of order~$\frac{\lambda}{n}$, the graph 
$$\Gamma=\cay(G\times U,S),~\text{where}~S=T^\#\cup H^\#~\text{and}~H=N\times U,$$ 
is a Neumaier graph with parameters $(n\lambda^2,(n+1)\lambda-2,\lambda-2,1,\lambda)$. Moreover, $\Gamma$ is strongly regular if and only if $n=\lambda$; in this case $\Gamma$ has parameters $(n^3,n^2+n-2,n-2,n+2)$.
\end{lemm}

\begin{proof}
 Clearly, $|G\times U|=|G||U|=m\lambda$, $|H|=|N||U|=\lambda$, and $|S|=|T^\#|+|H^\#|=k+\lambda-2$. Observe that $H$ is a clique in $\Gamma$ because $H^\#\subseteq S$. Let $g\in (G\times U)\setminus H$. The number of the neighbours of $g$ in $H$ is equal to $|Sg\cap H|=|(T^\#\cup H^\#)g\cap H|=|(T^\#g)\cap H|$. Since $T$ is semiregular, $T$ is a transversal for $N$. So $|(T^\#g)\cap H|=1$ for every $g\in (G\times U)\setminus H$. Therefore $H$ is $1$-regular clique. 

It remains to verify that $\Gamma$ is edge-regular with parameter~$\lambda-2$ or, equivalently, every element of $S$ enters the element $\underline{S}^2$ with coefficient~$\lambda-2$. A straightforward computation in the group ring of $G$ implies that

\begin{equation}\label{ssquare}
\begin{split}
\underline{S}^2&=(\underline{T}^\#+\underline{H}^\#)^2=\\
            &=(\underline{T}+\underline{H}-2e)^2=\\
            &=\underline{T}^2+\underline{H}^2+4e+\underline{T}\cdot\underline{H}+\underline{H}\cdot\underline{T}-4\underline{T}-4\underline{H}=\\
            &=ke+\lambda(\underline{G}-\underline{N})+\lambda\underline{H}+4e+2 \underline{(G\times U)}-4\underline{T}-4\underline{H}=\\
            &=(k+\lambda-2)e+(\lambda-2)\underline{S}+(\lambda+2)(\underline{G^\#\setminus (N^\# \cup T^\#)})+2(\underline{(G\setminus N)U^\#})
\end{split}
\end{equation}
as required. It should be mentioned that in the above computation, we use Eq.~\eqref{trivial} and the equalities $\underline{T}\cdot\underline{H}=\underline{H}\cdot\underline{T}=\underline{(G\times U)}$ which hold because $T$ is a transversal for $N$ in $G$ and hence $\underline{T}\cdot\underline{N}=\underline{N}\cdot\underline{T}=\underline{G}$. 

Eq.~\eqref{ssquare} yields that $\Gamma$ is strongly regular if and only if $\lambda+2=2$ or one of the sets $G^\#\setminus (N^\#\cup T^\#)$, $U^\#(G\setminus N)$ is empty. Obviously, $\lambda>0$ and $G^\#\setminus (N^\#\cup T^\#)\neq \varnothing$. Therefore $\Gamma$ is strongly regular if and only if $U^\#(G\setminus N)=\varnothing$. The latter condition is equivalent to $|U|=\frac{\lambda}{n}=1$ and we are done. 
\end{proof}

RSRDSs from Lemma~\ref{hadamar} satisfy the condition of Lemma~\ref{neumaier} and hence they can be used for constructing Neumaier Cayley graphs. Table~$2$ from Corollary~\ref{neunew} collects families of strictly Neumaier Cayley graphs obtained in such way.

\begin{corl}\label{neunew}
For each group $G$ from the first column of Table~$2$, there is a strictly Neumaier Cayley graph with parameters $(v,k,\lambda,m,s)$ from the second column of Table~$2$ over $G\times U$, where $U$ is an arbitrary group of order~$\frac{s(k-\lambda)}{|G|}$. 
\end{corl}

\begin{table}[h]
\centering
{\small
\begin{tabular}{|l|l|l|}
  \hline
  group $G$ & parameters $(v,k,\lambda,m,s)$ & comment   \\
  \hline
  $H_{2r+1}(q)$ & $(q^{4r-1},q^{2r}+q^{2r-1}-2,q^{2r-1}-2,1,q^{2r-1})$ &  $q$ is an odd prime power, $r\geq 1$\\ \hline
  $E^2_{2r+1}(p)$ & $(p^{4r-1},p^{2r}+p^{2r-1}-2,p^{2r-1}-2,1,p^{2r-1})$  & $p$ is an odd prime, $r\geq 1$ \\  \hline
  $G_0\times G_0\times C_2^j$ & $(2^{4i-j},2^{2i}+2^{2i-j}-2,2^{2i-j}-2,1,2^{2i-j})$  &  $G_0$ is abelian, $|G_0|=2^i$, $i\geq j\geq 1$\\  \hline
	$C_4^l$ & $(2^{4l-2},2^{2l}+2^{2l-2}-2,2^{2l-2}-2,1,2^{2l-2})$  &  $l\geq 3$\\  \hline
  $G_0\times C_2$ & $(8u^4,6u^2-2,2u^2-2,1,2u^2)$  & $G_0\in \mathcal{E}$, $|G_0|=4u^2$ \\  \hline
\end{tabular}
}
\caption{Strictly Neumaier Cayley graphs}
\end{table}

\begin{rem}
It should be noted that the construction given in the above proposition is a special case of the constructions given in~\cite[Section~3]{EGKM21} and~\cite{GK19}. Indeed, each of the graphs $\Gamma_u$, $u\in U$, with vertex set $Gu$ and edge set $\{(gu,tgu):~g\in G,~t\in T^\#\}$ admits a spread of $1$-regular cocliques corresponding to the $N$-cosets and two vertices of $\Gamma$ are adjacent if and only if they belong to the same $\Gamma_u$ and adjacent in $\Gamma_u$ or belong to the corresponding cocliques of $\Gamma_{u_1}$ and $\Gamma_{u_2}$, $u_1,u_2\in U$ (possibly, $u_1=u_2$). 
\end{rem}

The corollary below immediately follows the definition of the class of groups $\mathcal{E}$ and Corollary~\ref{neunew}.

\begin{corl}\label{divisor}
For every prime $p$, there exists a strictly Neumaier Cayley graph whose number of vertices is divisible by~$p$. 
\end{corl}

\subsection{Construction~$2$}\label{s:relative2}

Let $A_0$ and $H_0$ be nontrivial abelian groups, $G_0=A_0\times H_0$, $S_0$ an identity-free inverse-closed subset of $G_0$ such that $S_0\cap A_0=\varnothing$, and $\Gamma_0=\cay(G_0,S_0)$. Clearly, $\Gamma_0$ is $k_0$-regular graph with $v_0$ vertices, where $k_0=|S_0|$ and $v_0=|G_0|$.

Let $A$ be an abelian group such that $A\geq A_0$ and $|A:A_0|=2$, $C$ a cyclic group of order~$2$, $c$ a nontrivial element of $C$, $H=H_0\times C$, and $G=A\times H$. Suppose that $m_0=|H_0|$ is even and $H$ has an RSRDS $T$ with forbidden subgroup $C$ and parameters~$(m_0,2,m_0,m_0/2)$. Observe that $T$ is a transversal for $C$ in $H$. Put
$$S=(A\setminus A_0)T\cup S_0C\cup\{c\}.$$
Since $S_0$ is inverse-closed and identity-free, $S$ so is and hence 
$$\Gamma=\cay(G,S)$$
is a Cayley graph. The graph $\Gamma$ is $k$-regular graph with $v$ vertices, where $k=|S|$ and $v=|G|$. It easy to compute that 
$$v=|G|=|A||H|=4|A_0||H_0|=4|G_0|=4v_0$$ 
and 
$$k=|S|=|A\setminus A_0||T|+|S_0||C|+1=|A_0||H_0|+|S_0||C|+1=v_0+2k_0+1,$$
where the third equality holds because $|A:A_0|=2$ and $T$ is a transversal for $C\cong C_2$ in $H=H_0\times C$.

The graph $\Gamma$ is a union of the graphs $\Gamma_1=\cay(G,S_1)$ and $\Gamma_2=\cay(G,S_2)$, where $S_1=(A\setminus A_0)T$ and $S_2=S_0C\cup\{c\}$. It is easy to see that the graph $\Gamma_1$ is the tensor product of the graphs $\Gamma_{11}=\cay(A,A\setminus A_0)$ and $\Gamma_{12}=\cay(H,T)$. The first of them is a complete bipartite graph with two parts of size $|A_0|$, whereas the second one is a distance-regular antipodal graph of diameter~$3$ (see Lemma~\ref{delta}). The graph $\Gamma_2$ is a disjoint union of two $2$-clique extensions of $\Gamma_0$.

\begin{lemm}\label{c2l1}
If $H_0$ is an $(m_0/2)$-regular clique of size~$m_0$ in $\Gamma_0$, then $H$ is $(m/2)$-regular clique of size~$m$ in $\Gamma$, where $m=2m_0$.
\end{lemm}

\begin{proof}
Let $H_0$ be a clique of $\Gamma_0$. Then $H_0^\#\subseteq S_0$. By the definition of $S$, we have $H^\#\subseteq S$ and hence $H$ is a clique of $\Gamma$. Suppose that $H_0$ is an $(m_0/2)$-regular clique in $\Gamma_0$ and then let us show that $H$ is an $m/2=m_0$-regular clique in $\Gamma$. Each element $g\in G\setminus H$ can be presented in the form $g=ah$, where $a\in A^\#$ and $h\in H$. If $a\in A_0^\#$, then the number $|Sg\cap H|$ of neighbours of $g$ in $H$ is equal to
$$|Sg\cap H|=|S_0Ca\cap H|=|(S_0a\cap H_0)C|=m_0.$$
If $a\in A\setminus A_0$, then 
$$|Sg\cap H|=|(A\setminus A_0)Ta\cap H|=|T|=m_0.$$
Thus, $H$ is an $(m/2)$-regular clique.
\end{proof}

\begin{lemm}\label{c2l2}
Suppose that $\Gamma_0$ is $(v_0,k_0,\lambda_0)$-edge-regular. Then $\Gamma$ is edge-regular if and only if 
\begin{equation}\label{criterion}
v_0=4(k_0-\lambda_0-1).
\end{equation} 
If the latter is the case, then the parameters of $\Gamma$ are $(v,k,\lambda)$, where $\lambda=2k_0$. 
\end{lemm}

\begin{proof}
Let us compute $\underline{S}^2$. Recall that $S=S_1\cup S_2$, where $S_1=(A\setminus A_0)T$ and $S_2=S_0C\cup\{c\}$. Clearly,
\begin{equation}\label{eq0}
\underline{S}^2=\underline{S_1}^2+\underline{S_2}^2+2\underline{S_1}\cdot\underline{S_2}.
\end{equation}
Each summand in the right-hand side of the above equality will be computed separately. In the computations further, we freely use Eq.~\eqref{trivial}.
Firstly,
\begin{equation}\label{eq1}
\begin{split}
\underline{S_1}^2&=(\underline{A}-\underline{A_0})^2\cdot\underline{T}^2=\\
                 &=\underline{A_0}^2\cdot(m_0e+(m_0/2)(\underline{H}-\underline{C}))=\\
                 &=|A_0|\underline{A_0}\cdot(m_0e+(m_0/2)\underline{H_0^\#C})=\\
                 &=m_0|A_0|\underline{A_0}+|A_0|(m_0/2)\underline{A_0H_0^\#C}=\\
                 &=v_0\underline{A_0}+(v_0/2)(\underline{G_0}-\underline{A_0})\cdot \underline{C},
\end{split}
\end{equation}
where the second equality holds because $|A:A_0|=2$ and $T$ is an RSRDS with forbidden subgroup $C$ and parameters~$(m_0,2,m_0,m_0/2)$. 
Secondly,
\begin{equation}\label{eq2}
\begin{split}
\underline{S_2}^2&=(\underline{S_0C}+c)^2=\\
                 &=\underline{S_0}^2\cdot \underline{C}^2+2c\underline{S_0C}+e=\\
                 &=2(k_0e+\lambda_0\underline{S_0}+\xi)\cdot \underline{C}+2\underline{S_0C}+e=\\
                 &=(2k_0+1)e+2k_0c+2(\lambda_0+1)\underline{S_0C}+2\xi\underline{C},
\end{split}
\end{equation}
where $\xi\in \mathbb{Z}G$ is such that the support of $\xi$ lies in $G_0^\#\setminus S_0$ and the third equality holds because $\Gamma_0$ is edge-regular and hence $|S_0g_0\cap S_0|=\lambda_0$ for every $g_0\in S_0$. Finally,
\begin{equation}\label{eq3}
\begin{split}
2\underline{S_1}\cdot\underline{S_2}&=2(\underline{A}-\underline{A_0})\cdot\underline{T}\cdot(\underline{S_0C}+c)=\\
                                     &=2(\underline{A}-\underline{A_0})\cdot\underline{T}\cdot\underline{C}\cdot\underline{S_0}+2(\underline{A}-\underline{A_0})\cdot\underline{T}c=\\
                                     &=2(\underline{A}-\underline{A_0})\cdot\underline{H}\cdot \underline{S_0}+2(\underline{A}-\underline{A_0})\cdot(\underline{H}-\underline{T})=\\
                                     &=2(\underline{G}-\underline{G_0\times C})\cdot \underline{S_0}+2(\underline{A}-\underline{A_0})\cdot(\underline{H}-\underline{T})=\\
                                     &=2k_0(\underline{G}-\underline{G_0\times C})+2(\underline{A}-\underline{A_0})\cdot(\underline{H}-\underline{T})=\\
                                     &=2k_0\underline{S_1}+2(k_0+1)(\underline{A}-\underline{A_0})\cdot(\underline{H}-\underline{T}),
\end{split}
\end{equation}
where the third equality holds because $T$ is a transversal for $C$ in $H$. From Eqs.~\eqref{eq0}-\eqref{eq3} and $S_0\cap A_0=\varnothing$ it follows that $\underline{S_1}$ and $\underline{c}$ enter $\underline{S}^2$ with coefficient $2k_0$, whereas $\underline{S_0C}=\underline{S}-\underline{S_1}-\underline{c}$ enters $\underline{S}^2$ with coefficient $v_0/2+2(\lambda_0+1)$. Therefore $\Gamma$ is edge-regular if and only if Eq.~\eqref{criterion} holds and if the latter is the case, then the parameter $\lambda$ is equal to $2k_0$. 
\end{proof}

Lemmas~\ref{c2l1} and~\ref{c2l2} imply the following statement.

\begin{corl}\label{c2cor1}
If $\Gamma_0$ is a Neumaier graph with parameters~$(v_0,k_0,\lambda_0,m_0/2,m_0)$ and $v_0=4(k_0-\lambda_0-1)$, then $\Gamma$ is a Neumaier graph with parameters~$(v,k,\lambda,m/2,m)$, where
$$v=4v_0,~k=v_0+2k_0+1,~\lambda=2k_0,~m=2m_0.$$
\end{corl}

\begin{lemm}\label{c2l3}
If $\Gamma_0$ is strongly regular, then $\Gamma$ is not strongly regular. 
\end{lemm}

\begin{proof}
Let $\lambda_0$ and $\mu_0$ be the numbers of common neighbours of two abjacent and two distinct non-adjacent vertices of $\Gamma_0$, respectively. Then $\underline{S_0}^2=k_0e+\lambda_0\underline{S_0}+\mu_0(\underline{G}^\#-\underline{S_0})$ and hence $\xi$ in Eq.~\eqref{eq2} is equal to $\mu_0(\underline{G}^\#-\underline{S_0})$. Due to Eqs.~\eqref{eq0}-\eqref{eq3} and $S_0\cap A_0=\varnothing$, we conclude that $\underline{A_0}^\#$ and $\underline{A_0^\#c}$ enter $\underline{S}^2$ with distinct coefficients $v_0+2\mu_0$ and $2\mu_0$, respectively. Thus, $\Gamma$ is non-strongly regular.
\end{proof}

\begin{proof}[Proof of Theorem~\ref{main2}]
Let $n$ be a positive even integer, $A_0\cong H_0\cong C_2^n$ and hence $G_0\cong C_2^{2n}$, and $l=2^{n-1}+1$. It is well-known that there exist pairwise trivially intersecting and trivially intersecting with $A_0$ subgroups $A_1,\ldots,A_l$ of $G_0$ of order~$2^n$. Put 
$$S_0=\bigcup \limits_{i=1,\ldots,l} A_i^\#.$$
Then the graph $\Gamma_0=\cay(G_0,S_0)$ is strongly regular with parameters
$$(v_0,k_0,\lambda_0,\mu_0)=(2^{2n},(2^n-1)(2^{n-1}+1),2(2^{n-1}-1)(2^{n-2}+1),2^{n-1}(2^{n-1}+1))$$
(see, e.g.~\cite{BM2022}). Clearly, $S_0\cap A_0=\varnothing$. Due to the definition of $\Gamma_0$, each of the subgroups $A_1,\ldots,A_l$ is a $2^{n-1}$-regular clique of size~$2^n$ in $\Gamma_0$. So $\Gamma_0$ is a Neumaier graph. 

Let $A$ be an abelian group of order $2^{n+1}$ such that $A\geq A_0$ and $|A:A_0|=2$, $C\cong C_2$, $H=H_0\times C\cong C_2^{n+1}$, and $G=A\times H$. By Lemma~\ref{hadamar} (Lines~3 and~5 in Table~1), the group $H$ has an RSRDS~$T$ with parameters~$(2^n,2,2^n,2^{n-1})$ and forbidden subgroup $C$. Therefore one can construct the graph $\Gamma=\cay(G,S)$ with connection set $S=(A\setminus A_0)T\cup S_0C\cup\{c\}$. It is easy to verify that the parameters of $\Gamma_0$ satisfy Eq.~\eqref{criterion} and hence $\Gamma$ is a Neumaier graph with the required parameters by Corollary~\ref{c2cor1}. By Lemma~\ref{c2l3}, $\Gamma$ is a strictly Neumaier graph.
\end{proof}

\section{An algorithm for enumeration of Neumaier Cayley graphs with a spread given by the cosets of a subgroup}\label{s:enum}

In this section, we provide Algorithm~\ref{alg:one} for enumerating Neumaier Cayley graphs with a spread given by the cosets of a subgroup. The running time of this algorithm is exponential. A correctness of Algorithm~\ref{alg:one} follows from Proposition~\ref{correct}. Before we give a proof of this proposition, we prove an auxiliary lemma.

\begin{lemm}\label{lem:StabHg_i}
    Let $G$ be a finite group, $H\leq G$, and $G/H=\{Hg_1,\dots,Hg_n\}$, where $g_1=e$. Then 
    $
    \operatorname{Stab}_{\operatorname{Aut}G}(Hg_i) \leq \operatorname{Stab}_{\operatorname{Aut}G}(H)
    $
   for every $i \in \{2,\ldots,n\}$, where the stabilisers are setwise.
\end{lemm}
\begin{proof}
    Let $i \in \{2,\ldots,n\}$, $h\in H$, and $\varphi \in \operatorname{Stab}_{\operatorname{Aut}G}(Hg_i)$. The latter condition implies that there exists $h^\prime\in H$ such that $\varphi(hg_i)=h^\prime g_i$ and hence $\varphi(h)\varphi(g_i)=h^\prime g_i$. Therefore $\varphi(h)=h^\prime g_i\varphi(g_i)^{-1}$. Since $g_i$ and $\varphi(g_i)$ are representatives of the coset $Hg_i$, we have $g_i\varphi(g_i)^{-1} \in H$. Thus, $\varphi(h) \in H$, and, consequently, $\varphi \in \operatorname{Stab}_{\operatorname{Aut}G}(H)$.
\end{proof}


\begin{prop}\label{correct}
For any valid instance of input data, Algorithm \ref{alg:one} finishes in finite time and the following statements hold.
\begin{enumerate}
\tm{1} The set $\mathcal{N}$ consists of Neumaier Cayley graphs over $G$ with parameters $(v,k,\lambda,m,s)$ and a spread given by the $H$-cosets.

\tm{2} Every Neumaier Cayley graph over $G$ with parameters $(v,k,\lambda,m,s)$ and a spread given by the $H$-cosets is isomorphic to some graph from $\mathcal{N}$.
\end{enumerate}
\end{prop}

\begin{proof}
The set $T$ obtained before the last step of the algorithm satisfies Condition~(1) from Theorem~\ref{main1} by construction. So the graph added to the set $\mathcal{N}$ at the last step of the algorithm is a Neumaier graph with the required parameters by Theorem~\ref{main1}. 

Let us prove that every Neumaier graph $\cay(G,H^\#\cup T)$ with parameters $(v,k,\lambda,m,s)$ and a spread given by the $H$-cosets is isomorphic to some graph from $\mathcal{N}$. Let $T_2=T\cap Hg_2$. Theorem~\ref{main1}(1) yields that $|T_2|=m$. By the definition of the set $S_2$, there exist $T_2^\prime\in S_2$ and $\varphi\in \operatorname{Stab}_{\operatorname{Aut}G}(Hg_2)$ such that $\varphi(T_2)=T_2^\prime$. The algorithm constructs all Neumaier Cayley graphs $\cay(G,H^\#\cup T^\prime)$ with parameters $(v,k,\lambda,m,s)$ and a spread given by the $H$-cosets such that $T^\prime\cap Hg_2=T_2^\prime$. Together with Lemma~\ref{lem:StabHg_i} and condition $\varphi\in \operatorname{Stab}_{\operatorname{Aut}G}(Hg_2)$, this implies that 
$$\cay(G,\varphi(H^\#\cup T))=\cay(G,\varphi(H^\#)\cup \varphi(T_2)\cup \varphi(T\setminus T_2))=\cay(G,H^\#\cup T_2^\prime\cup \varphi(T\setminus T_2))\in \mathcal{N}$$ 
and we are done.
\end{proof}

\begin{algorithm}\label{Algorithm}
\caption{Enumeration of Neumaier Cayley graphs}\label{alg:one}
\KwData{~\\
a feasible tuple of parameters $(v,k,\lambda,m,s)$ of a Neumaier graph\; 
a finite group $G$, $|G| = v$\; 
a subgroup $H<G$, $|H| = s$, with cosets $Hg_1,\dots,Hg_n$, where $g_1=e$\; 
the automorphism group $\operatorname{Aut}G$\;
the stabiliser $\operatorname{Stab}_{\operatorname{Aut}G}(Hg_2)$ \Comment*[r]{See Lemma \ref{lem:StabHg_i}}}
\KwResult{~\\the set $\mathcal{N}$ of Neumaier Cayley graphs with parameters $(v,k,\lambda,m,s)$ over $G$ (each graph is given as $\operatorname{Cayley}(G,H^\# \cup T)$)
\Comment*[r]{See Theorem \ref{main1}}
}
~\\
$\mathcal{N}\gets \varnothing$\\
$S_2 \gets m\text{-subsets of } Hg_2$ that are non-equivalent under $\operatorname{Stab}_{\operatorname{Aut}G}(Hg_2)$\;

\For{$T_2 \in S_2$}{
    \If{$T_2^{-1} \cap (Hg_2 \setminus T_2) \ne \varnothing$}{
        take next $T_2$\;
     }
     $R_3 \gets Hg_3 \cap T_2^{-1}$ \Comment*[r]{The elements from $Hg_3$ given by $T_2$}
     $S_3 \gets (m-|R_3|)\text{-subsets of } Hg_3$\;
     \For{$\overline{T_3} \in S_3$}{
     $T_3 \gets \overline{T_3} \cup R_3$\;
    \If{$T_3^{-1} \cap ((Hg_2 \cup Hg_3) \setminus (T_2 \cup T_3)) \ne \varnothing$}{
        take next $\overline{T_3}$\;
        }
        $R_4 \gets Hg_4 \cap (T_2^{-1} \cup T_3^{-1})$ \Comment*[r]{The elements from $Hg_4$ given by $T_2 \cup T_3$}
     \If{$|R_4| > m$}{
        take next $\overline{T_3}$\;
     }
     $S_4 \gets (m-|R_4|)\text{-subsets of } Hg_4$\;
        $\ldots$\\
        \For{$\overline{T_n} \in S_n$}{
     $T_n \gets \overline{T_n} \cup R_n$\;
    \If{$T_n^{-1} \cap ((Hg_2 \cup \ldots \cup Hg_n) \setminus (T_2 \cup \ldots \cup T_n)) \ne \varnothing$}{
        take next $\overline{T_n}$\;
        }
        $T \gets T_2 \cup \ldots \cup T_n$\;
        \If{$T = T^{-1}$ and Conditions (2), (3) from Theorem \ref{main1} hold}{
            $\mathcal{N}\gets \mathcal{N}\cup \cay(G,H^\# \cup T)$\;
        }
        }
        $\ldots$\\
        
    }
}
\end{algorithm}

Algorithm~\ref{alg:one} was realized using GAP and MAGMA. Applying Algorithm~\ref{alg:one}, a standard procedure of isomorphism testing between two graphs implemented in MAGMA, and counting the number of common neighbours of two non-adjacent vertices in each graph, we enumerate all pairwise nonisomorphic strictly Neumaier Cayley graphs with a spread given by the cosets of a subgroup on at most 64 vertices. It turns out that most of the computed strictly Neumaier graphs have nexus~1 and there are only few strictly Neumaier graphs with nexus greater than~1. All of the latter graphs are collected in Table~\ref{table:1}. The smallest graph from Theorem~\ref{main2} with parameters $(64,35,18,4,8)$ was found by a computer calculation using Algorithm~\ref{alg:one}.

\begin{table}
\begin{center}
\begin{tabular}{|c|c|}
  \hline
  Tuple & \#SNGs\\
	\hline
	(16,9,4,2,4)  & 1\\
  \hline
  (64,28,12,3,8)  & 18 \\
  \hline
  (64,35,18,4,8)  & 1380 \\
  \hline
  (64,42,26,5,8)  & 1 \\
  \hline
\end{tabular}~~~~~~~~
\end{center}
\caption{Strictly Neumaier Cayley graphs with a spread given by the cosets of a subgroup and nexus $>1$ on at most 64 vertices}\label{table:1}
\end{table}

 \newpage

\section*{Acknowledgements}
All authors are indexed alphabetically according to the convention of the mathematics society. 
All authors are co-first authors.
Rhys J. Evans acknowledges the support of the Slovenian Research and Innovation Agency (ARIS), project numbers P1-0294 and J1-4351. S. Goryainov was supported by the grant of The Natural Science Foundation of Hebei Province (project No.~A2023205045).
Da Zhao acknowledges the support by the National Natural Science Foundation of China (No. 12471324, No. 12501459, No. 12571353), and the Natural Science Foundation of Shanghai, Shanghai Sailing Program (No. 24YF2709000).

\end{document}